\documentclass{amsart}
\usepackage{url} 
\usepackage{amsmath,amssymb,amsthm}
\usepackage{verbatim}
\usepackage{enumerate}
\numberwithin{equation}{section} 

\DeclareMathOperator{\gal}{Gal}
\DeclareMathOperator{\ord}{ord}
\DeclareMathOperator{\car}{char}

\DeclareMathOperator{\gl}{GL}
\DeclareMathOperator{\SL}{SL}

\DeclareMathOperator{\agl}{AGL}
\DeclareMathOperator{\id}{Id}

\newtheorem{thm}{Theorem}[section]
\newtheorem{lem}[thm]{Lemma}
\newtheorem{cor}[thm]{Corollary}
\newtheorem{defn}[thm]{Definition}

\begin{document}

\title{On the pre-image of a point under an isogeny and Siegel's theorem}
\author{Jonathan Reynolds}
\address{Mathematisch Instituut \\ Universiteit Utrecht \\ Postbus 80.010 \\ 3508 TA Utrecht \\ Nederland}
\email{J.M.Reynolds@uu.nl}
\date{\today}
\thanks{The author is supported by a Marie Curie Intra European Fellowship (PIEF-GA-2009-235210)}
\subjclass[2000]{11G05, 11A51}

\begin{abstract}  Consider a rational point on an elliptic curve under an isogeny. Suppose that the action of Galois partitions the set of its pre-images into $n$ orbits. It is shown that all such points above a certain height have their denominator divisible by at least $n$ distinct primes. This generalizes Siegel's theorem and more recent results of Everest et al. For multiplication by a prime $l$, it is shown that if $n>1$ then either the point is $l$ times a rational point or the elliptic curve emits a rational $l$-isogeny.                 
\end{abstract}

\maketitle


\section{Introduction}

Let $(E,O)$ denote an elliptic curve defined over a number field $K$ with Weierstrass coordinate functions $x,y$. Siegel \cite{Sie01} proved that there are only finitely many $P \in E(K)$ with $x(P)$ belonging to the ring of integers $\mathcal{O}_K$. Given a finite set $S$ of prime ideals of $\mathcal{O}_K$, the ring of $S$-integers in $K$ is
\[
\mathcal{O}_{KS} := \{ x \in K : \ord_{ \mathfrak{p}}(x) \ge 0 \textrm{ for all } \mathfrak{p} \notin S   \}.
\]   
Mahler \cite{Mahler} conjectured that there are finitely many $P \in E(K)$ with $x(P) \in \mathcal{O}_{KS}$ and proved his conjecture for $K= \mathbb{Q}$. Lang \cite{MR0130219} gave a modernized exposition and proved Mahler's conjecture 
for number fields. A corollary to this is that there are finitely many $P \in E(K)$ with $f(P) \in \mathcal{O}_{KS}$, where $f \in K(E)$ is any function having a pole at $O$  
(see Corollary 3.2.2 in Chapter IX of~\cite{MR2514094}). It is unknown how much further these $S$-integral points can be generalized before finiteness fails. For example, in \cite{MR2548983} it is suggested that, in rank one subgroups, only the size of $S$ has to be fixed and not the primes in the set.

Everest, Miller and Stephens \cite{MR2045409} proved conditionally for $K=\mathbb{Q}$ that there are finitely many multiples $mP$ of a non-torsion point $P$ which have the denominator of $x(mP)$ divisible by a single prime not belonging to a fixed set. These denominators generate an elliptic divisibility sequence, a genus-$1$ analogue of more classical sequences such as Fibonacci or Mersenne, and the condition, which they called \emph{magnified}, is that the non-torsion point $P$ has a preimage defined in a number field of degree less than the degree of the isogeny (see Definition \ref{1.1}). The finiteness result concerning primes in elliptic divisibility sequences was generalized to number fields under an extra assumption that the pre-image lie in a Galois extension \cite{MR2164113}. In what follows this extra assumption is removed, there is no restriction to rank one subgroups and, analogous to the results for integral points, $S$ and $f$ are arbitrary (see Theorem \ref{gali}). Moreover, using the division polynomials of $E$, the magnified condition is replaced with a factorization criterion which can be checked more readily (see Theorem \ref{step}). This leads to the first known proof that, as conjectured by Everest, the condition can fail for prime degrees. In particular,
either the magnified point is $l$ times a rational point or the elliptic curve emits a rational $l$-isogeny for some prime $l$ (see Theorem \ref{1.2} and Corollary \ref{1.3}).                

\subsection{Division Polynomials} Let $E$ be an elliptic curve defined over a field $K$ with Weierstrass coordinate functions $x,y$. For any integer $m \in \mathbb{Z}$, the $m$th division polynomial of $E$ is the polynomial 
$\psi_m \in K[x,y] \subset K(E)$ as given on p. 39 of \cite{MR1771549}.
Moreover, $\psi_m^2 \in K[x]$ and there exists $\theta_m \in K[x]$ such that
\[
[m]x=\frac{\theta_m}{\psi_m^2}.
\]
Given $P \in E(K)$, define $\delta_m^P \in K[x]$ by
\[
\delta_m^P= \left\{ \begin{array}{cc} \theta_m-x(P)\psi_m^2 & \textrm{if } P \ne O \\ \psi_m^2 & \textrm{otherwise.} \end{array} \right.
\]

The zeros of $\delta_m^P$ give the values of $x(R)$ for which $mR=P$.

\begin{thm} \label{gali} Let $K$ be a number field, $S$ a finite set of prime ideals of $\mathcal{O}_K$ and $f \in K(E)$ a function having a pole at $O$. Suppose that $\delta_m^P$ has $n$ factors over $K$ for some $P \in E(K)$. Then for all such points of sufficiently large height,
\begin{equation} \label{pri}
\{ \textrm{primes } \mathfrak{p} \notin S : \ord_{\mathfrak{p}}(f(P))<0 \}
\end{equation}
contains at least $n$ distinct primes.
\end{thm}

By Siegel's theorem, along with the generalizations of it by Mahler and Lang, (\ref{pri}) contains at least one prime for all $P \in E(K)$ of sufficiently large height. So Theorem~\ref{gali} extends Siegel's result whenever $\delta_m^P$ factorizes for some non-torsion point $P$.  In Section \ref{proof1} it shown that the points not of sufficiently large height are $m$ times a $U$-integral point for some finite set $U$ of prime ideals of $\mathcal{O}_L$, where $U$ and $L$ are given explicitly. Quantitative results for the number of such points can be found using \cite{MR1328329}.

In addition to being conjectured finite \cite{MR2164113, MR2045409, MR2365225}, the number of prime terms in an elliptic divisibility sequence coming from a minimal Weierstrass equation is believed to be uniformly bounded \cite{MR2429645, Mah}. Similarly, the number of terms without a primitive divisor is believed to be uniformly bounded \cite{MR2486632, MR2301226, IngrSilv}. There are also links between primitive divisors and extensions of Hilbert's tenth problem \cite{MR2377127, MR2480276}. However, most results in these directions have also used that $\delta_m^P$ factorizes for some $m$. Therefore it seems reasonable to give a detailed study of this condition.    

Let $K$ be a number field, $E/K$ an elliptic curve and suppose that $\delta_m^P$ factorizes for some non-torsion point $P \in E(K)$. Then, since $\delta_m^P$ is monic and has degree $m^2$, $mR=P$ with $[K(R):K] \le m^2/2$. Assuming Lehmer's conjecture (see \cite{MR2029512}), there exists a constant $c$ depending only on $E$ and $K$ so that $\hat{h}(P)=m^2\hat{h}(R)>2c$. However, since the constant is unknown, Lehmer's conjecture gives no way of knowing if the condition will fail for a given point or not. For prime degrees this issue is resolved by the following

\begin{thm} \label{1.2}
Let $l$ be a prime, $E$ an elliptic curve defined over a field $K$ with $\car{K} \nmid l$ and $P$ a $K$-rational point on $E$. Then either 
\begin{enumerate}[i.]
\item $\delta_l^P$ is irreducible, or
\item $E$ emits a $K$-rational $l$-isogeny, or
\item $[l]^{-1}P$ contains a $K$-rational point.
\end{enumerate}  
\end{thm} 

Given an elliptic curve $E/ \mathbb{Q}$, the set of all curves $E'$ isogenous to $E$ over $\mathbb{Q}$ is finite (up to 
isomorphism) and is known as an isogeny class. V\'{e}lu's formulae \cite{MR0294345} and the Weierstrass parameterization of
the elliptic curve can be used to find an isogeny class. This is best illustrated in an algorithm developed by Cremona 
\cite{MR1628193}. He has used his algorithm to produce tables of isogeny classes \cite{Credat}. For each curve in the class, 
non-torsion generators of the Mordell-Weil group are also given. 
For a number field the primes which can occur as orders of isogenies have been well studied \cite{MR931181}. Applying a famous result of Mazur \cite{MR482230} gives 

\begin{cor} \label{1.3}
Let $K=\mathbb{Q}$ and $P \in E(\mathbb{Q})$. If $\delta_l^P$ factorizes for some prime $l$ then either $P$ is $l$ times a rational point, or $l \le 19$, or l=$37$, $43$, $67$, or $163$.
\end{cor}
The criterion in Corollary \ref{1.3} can readily be checked using, for example, $\mathtt{MAGMA}$ \cite{MR1484478} and so gives a way to determine if $\delta_l^P$ is irreducible for all primes $l$. What is known for composite $m$ is discussed in Section 5; note that if $\delta_m^P$ factorizes then $\delta_d^P$ does not necessarily factorize for some proper divisor $d>1$ of $m$, but counter-examples have only been found when $m=4$.
 
\section{The action of Galois on preimages} \label{action}

Let $E$ be an elliptic curve defined over a field $K$ with Weierstrass coordinate functions $x,y$. Given a Galois extension $L/K$, $\sigma \in \gal(L/K)$ and $R \in E(L)$, $\sigma(R)$ is defined by $\sigma(R)=(x(R)^{\sigma}, y(R)^{\sigma})$.  

\begin{defn}[\cite{MR2164113}] \label{1.1}
Let $K$ be a field, $E/K$ an elliptic curve, $P \in E(K)$ and $\phi:E' \to E$ an isogeny. Suppose that $E'$, $\phi$ and a point in $\phi^{-1}(P)$ are all defined over a finite extension $L/K$. If $[L:K]<\deg \phi$ then $P$ is called 
\emph{magnified}.
\end{defn}

Below (Theorem \ref{step}) it is shown that for a perfect field (which includes the applications referenced above) the magnified condition is equivalent to $\delta_m^P$ factorizing for some $m$.

\begin{lem} \label{2.1}
Assume that $\car{K} \ne 2$ or $K$ is perfect. Suppose that $P \in E(K)$ is not a $2$-torsion point, $E'/K$ is an elliptic curve with Weierstrass coordinate functions $x',y'$ and $\phi: E' \to E$ is an isogeny defined over $K$ with $\phi(R)=P$. Then $K(x'(R), y'(R))=K(x'(R))$. 
\end{lem}

\begin{proof}
Put $L=K(x'(R))$ and $L'=K(x'(R), y'(R))$. Then $[L':L] \le 2$. The assumptions on $K$ make $L'/L$ Galois. 
Suppose that $[L':L]=2$ and choose $\sigma$ to be the generator of $\gal(L'/L)$. Then $T=\sigma(R)-R$ is in the kernel of 
$\phi$ since $\sigma(\phi(R))-\phi(R)= O$. But $\sigma$ fixes $x'(R)$ so $R+T=\pm R$. Since $P$ is not a 
$2$-torsion point it follows that $\sigma(R)=R$ and $L'=L$.
\end{proof}        

\begin{lem} \label{2.2}
Assume that $K$ is perfect. If $P \in E(K) \setminus E[2]$ is magnified by an isogeny $\phi : E' \to E$ of degree $m$ then it is magnified by $[m]$.
\end{lem}

\begin{proof} Suppose that $E'$, $\phi$ and $Q \in \phi^{-1}(P)$ are all defined over a finite extension $L/K$ with 
$[L:K]<m$. The dual $\hat{\phi}: E \to  E'$ of $\phi$ is defined over $L$. Let 
$R \in \hat{\phi}^{-1}(Q)$. Lemma \ref{2.1} gives $L(x(R),y(R))=L(x(R))$. Now $f=x'\circ \hat{\phi} \in L(E)=L(x,y)$ is an
even function. Hence, $f \in L(x)$ and $f(x)=x'(Q)$ gives a polynomial in $L[x]$ whose roots determine the values of 
$x(R)$. Since $ \# \hat{\phi}^{-1}(Q) \le \deg{\hat{\phi}}=m$ and $K$ is perfect, this polynomial cannot have an irreducible factor of degree larger than $m$. Thus, $[L(x(R)):K]=[L(x(R)):L][L:K]<m^2$.   \end{proof}

\begin{thm} \label{step}
For $K$ a perfect field and an elliptic curve $E/K$, $P \in E(K)$ 
is magnified if and only if $\delta_m^P$ factorizes over $K$ for some $m$. 
\end{thm}

\begin{proof} 
If $P \in E[2]$ then $3P=P$ so $\delta_3^P$ factorizes. So assume that $P \notin E[2]$. By Lemma \ref{2.2}, $P$ is magnified if and only if it is magnified by $[m]$ for some $m>1$. The result now follows from Lemma \ref{2.1}.
\end{proof}

\section{Proof of Theorem \ref{gali}} \label{proof1}

\begin{proof}[Proof of Theorem \ref{gali}] Suppose firstly that $f$ is a Weierstrass $x$-coordinate function. Fix a set of generators of $E(K)/mE(K)$ and for every $P_j$ in the set, adjoin to $K$ the coordinates of the points in $[m]^{-1}P_j$. Note that this finite extension $L$ does not depend on $P$ and that the splitting field of $\delta_m^P$ is contained within it. Let $U$ be a finite set of prime ideals of $\mathcal{O}_L$ containing
\begin{itemize}
\item those which lie above the ideals in $S$,
\item those which the coefficients of the Weierstrass equation $U$-integers,
\item those which make $x(T)$ a $U$-integer for all non-zero $T \in E[m]$, and  
\item those which make $\mathcal{O}_{LU}$ a principal ideal domain.
\end{itemize}
By the Siegel-Mahler theorem we can assume that no $R \in [m]^{-1}P$ is $U$-integral. Write $x(R)=A_R/B_R^2$, where $A_R$ and $B_R$ are coprime in $\mathcal{O}_{LU}$. Then
\begin{equation} \label{x(P)}
x(P)=\frac{\theta_m(x(R))}{\psi_m^2(x(R))}=\frac{B_R^{2m^2}\theta_m \left( \frac{A_R}{B_R^2} \right) }{B_R^{2} \left( B_R^{2(m^2-1)}\psi_m^2 \left( \frac{A_R}{B_R^2} \right) \right)},
\end{equation}
where $B_R$ is coprime with the numerator. Let $R$ and $R'$ be two distinct points in $[m]^{-1}P$. Then $R'=R+T$ for some non-zero $T \in E[m]$.
From the addition formula it can be seen that $B_R$ and $B_{R'}$ are coprime in $\mathcal{O}_{LU}$. Any conjugate of a prime in the factorization of $B_R$ over $\mathcal{O}_{LU}$ divides the denominator of some element in the orbit $\{\sigma(x(R)) : \sigma \in \gal{(L/K)} \}$. Hence, using (\ref{x(P)}), the number of distinct prime ideals $\mathfrak{p} \notin S$ of $\mathcal{O}_K$ with $\ord_{\mathfrak{p}}(x(P))<0$ is at least equal to the number of factors of $\delta_m^P$ over $K$.

Finally, suppose that $f \in K(E)$ has a pole at $O$. We may assume that a Weierstrass equation for $E/K$ is of the the form $y^2$ equal to a monic cubic in $K[x]$. Now $f \in K(C)=K(x,y)$ and $[K(x,y):K(x)]=2$ give
\[
f(x,y)=\frac{\phi(x)+\psi(x)y}{\eta(x)},
\]
where $\phi(x), \psi(x), \eta(x) \in K[x]$. Now $\ord_O(\phi)=\ord_O(x^{\deg{\phi}})=-2 \deg{\phi}$. Similarly, $\ord_O(\psi)=-2 \deg{\psi}$ and $\ord_O(\eta)=-2 \deg{\eta}$. Since $O$ is a pole of $f$,
\[
\ord_O(f)=\ord_O(\phi+\psi y)-\ord_O(\eta)<0.
\] 
But $\ord_O(\phi+\psi y) \ge \min\{\ord_O(\phi),\ord_O(\psi)+\ord_O(y)$ 
and $\ord_O(y)=-3$, thus
\begin{equation} \label{pole}
2 \deg \eta < \max \{2 \deg{\phi}, 2\deg{\psi}+3  \}.
\end{equation}
Enlarge $S$ so that: 
\begin{itemize}
\item $\mathcal{O}_{KS}$ is a  a principal ideal domain;
\item the coefficients of the Weierstrass equation are $S$-integers;
\item $\phi(x), \psi(x), \eta(x) \in \mathcal{O}_{KS}[x]$ and their leading coefficients are $S$-units.
\end{itemize}
Assume that $x(P)y(P) \ne 0$ then $(x(P),y(P))=\left(A_P/B_P^2,C_P/B_P^3\right)$,
where $A_PC_P$ and $B_P$ are coprime in $\mathcal{O}_{KS}$. 
The condition (\ref{pole}) gives that $B_P$ divides the denominator and is coprime the numerator of $f(P)$ in $\mathcal{O}_{KS}$. Thus the result follows from the case $f=x$ above.   
\end{proof}  

\section{Proof of Theorem \ref{1.2}} \label{prime}

The condition that $\car K \nmid m$ ensures that multiplication by $m$ is separable and that $\#[m]^{-1}P=m^2$ (see 4.10 and 5.4 in Chapter III of \cite{MR2514094}). Hence, for $P \notin E[2]$ the splitting of field of $\delta_m^P$ is Galois over $K$. Note that $(\mathbb{Z}/m\mathbb{Z})^2$ is isomorphic to $E[m]$ and bijective with $[m]^{-1}P$. The actions of Galois on $E[m]$ and on $[m]^{-1}P$ are described by homomorphisms 
$\gal(\bar{K}/K) \to \gl_2(\mathbb{Z}/m\mathbb{Z})$ and $\gal(\bar{K}/K) \to \agl_2(\mathbb{Z}/m\mathbb{Z})$. Let $G_m$ and $\mathcal{G}_m$ be the images of these maps. Consider the homomorphism $\alpha_m: \mathcal{G}_m \to G_m$ given by $\alpha_m((A,v))=A$.	

\begin{lem} \label{m=2} 
Let $E$ be an elliptic curve defined over a field $K$ with $\car{K} \ne 2$ and let $P$ be a $K$-rational point on $E$. Then either 
\begin{enumerate}[i.]
\item $\delta_2^P$ is irreducible 
\item $P$ is a $2$-torsion point, or
\item $[2]^{-1}P$ has a $K$-rational point, or
\item $P$ is the image of a $K$-rational point under a $K$-rational $2$-isogeny. 
\end{enumerate} 
\end{lem}

\begin{proof} Let $2R=P$. Suppose that $P$ is not a $2$-torsion point. Using
Lemma \ref{2.1}, let $L=K(x(R),y(R))=K(x(R))$. If $\delta_2^P$ factorizes then we may choose $R$ so that $[L:K] \le 2$. If $[L:K]=1$ then we are in case iii. 
If $[L:K]=2$ then choose 
$\sigma \in \gal(L/K)$ to be non-trivial. Then $T=\sigma(R)-R$ is a $2$-torsion point since $\sigma(2R)-2R=\mathcal{O}$. Also $T \in E(K)$ since 
$\sigma(T)=-T$. Using this torsion point, we can construct an elliptic curve $E'/K$ and a $2$-isogeny $\phi: E \to E'$ 
with $\ker \phi=\{ \mathcal{O},T \}$ (see 8.2.1 of \cite{MR2312337}). Moreover, both $\phi$ and its dual $\hat{\phi}: E' \to E$ 
are defined over $K$. Put $\phi(R)=Q$. It follows that $\sigma(Q)=\phi(\sigma(R))=\phi(R+T)=\phi(R)$. Hence $Q \in E'(K)$ 
and $\hat{\phi}(Q)=P$.   \end{proof}

Note that, for $l=2$, Lemma \ref{m=2} is stronger than Theorem \ref{1.2}. 

\begin{proof}[Proof of Theorem \ref{1.2}]
If $P \in E[2]$ then we are in case ii or iii. So assume that $P \notin E[2]$. If $\# \ker \alpha_l >1$ then there exists a non zero $l$-torsion point $T$ and $\sigma \in \gal(\bar{K}/K)$ with $\sigma(R)=R+T$ for all $R \in [l]^{-1}P$. Hence $\tau \sigma \tau^{-1} (R)=R+\tau(T)$ for any $\tau \in \gal(\bar{K}/K)$. If $\tau(T) \in \left< T \right>$ for all $\tau \in \gal(\bar{K}/K)$ then we are in case ii (see 4.12 and 4.13 in Chapter III of \cite{MR2514094}). Otherwise, Galois acts transitively on $[l]^{-1}P$ and we are in case i.

Thus, it remains to consider the case where $\alpha_l: \mathcal{G}_l \to G_l$ is an isomorphism and, by Lemma \ref{m=2}, $l>2$. So $\alpha_l$ has an inverse $A \to (v \to Av+b_A)$ and the map $\beta_l: G_l \to E[l]$ given by $\beta_l(A)=b_A$ is a crossed homomorphism because
$\beta_l(AB)=Ab_B+b_A$. The map $\beta_l$ is said to be principal if for some fixed $v \in (\mathbb{Z}/l\mathbb{Z})^2$, $\beta_l(A)=Av-v$ for all $A \in G_l$. The group $H^1(G_l,E[l])$ is the quotient of the group of crossed homomorphisms $G_l \to E[l]$ and the group of  principal ones. If $l$ does not divide $\# G_l$ then the orders of $G_l$ and $E[l]$ are coprime, so it follows that $H^1(G_l,E[l])=0$.
So assume that $l \mid \# G_l$ and apply Proposition 15 of \cite{MR0387283}. Either $G_l$ is contained in a Borel subgroup and so we are in case ii since the span of a point of order $l$ is fixed by Galois, or $G_l$ contains $H_l=\SL_2(\mathbb{Z}/l\mathbb{Z})$. For the second possibility construct an inflation-restriction sequence as in the proof of Lemma 4 in \cite{MR1694286}.  Note that $H_l$ is normal since it is the kernel of the determinant on $G_l$. There is an exact sequence
\[ 
0 \to H^1(G_l/H_l,E[l]^{H_l}) \to H^1(G_l, E[l]) \to H^1(H_l, E[l]). 
\] 
For $l>2$ the first cohomology group is trivial since $E[l]^{H_l}$ is trivial. By Lemma~3 in \cite{MR1694286}, the third cohomology group is also trivial. Hence $H^1(G_1, E[l])=0$ and so $\beta_l$ must be principal. 
But then $-v=-Av+\beta_l(A)$ for all $A \in G_l$ gives a fixed point for the action on $[l]^{-1}P$ so we are in case iii.     
\end{proof}  

\section{Multiplication by a composite}

Let $\alpha_m$ be as in Section \ref{prime}. A result for all composite $m$ is  

\begin{thm} \label{nec}
Let $m>1$ be a composite integer, $E$ an elliptic curve defined over a field $K$ with $\car K \nmid m$  and $P$ a $K$-rational point on $E$. Then either  
\begin{enumerate}[i.]
\item $\delta_m^P$ is irreducible, or 
\item $\delta_d^P$ factorizes, where $d>1$ is a proper divisor of $m$, or
\item $E$ emits a $K$-rational $l$-isogeny for some prime $l \mid m$, or
\item $\alpha_m$ is an isomorphism.
\end{enumerate}  
\end{thm} 

\begin{proof}
If $\# \ker \alpha_m >1$ then there exists a non-zero $m$-torsion point $T$ and $\sigma \in \gal(\bar{K}/K)$ with $\sigma(R)=R+T$ for all $R \in [m]^{-1}P$.
If $T$ has order $d_1$ then write $d_1=ld_2$ where $l$ is prime. Now 
$\sigma^{d_2}R=R+d_2T$ for all $R \in [m]^{-1}P$.  
 Hence $\tau \sigma^{d_2} \tau^{-1} (R)=R+\tau(d_2T)$ for any $\tau \in \gal(\bar{K}/K)$.  
Assume that $\tau(d_2T)$ is not a multiple of $d_2T$ for some $\tau \in \gal(\bar{K}/K)$; otherwise we are in case iii. 
Then we can always find a Galois element which will take 
$R$ to $R+T_1$, where $T_1$ is any $l$-torsion point. Assume that $P \notin E[2]$ and $\delta_m^P$ factorizes over $K$. Let $R_1, R_2 \in [m]^{-1}P$ correspond to roots of two different factors. By assumption for any $T_1 \in E[l]$,  $R_2+T_1$ corresponds to a root of the same polynomial. Thus, $\rho (R_1)-R_2$ is not a $l$-torsion point for any $\rho \in \gal(\bar{K}/K)$. So $\rho(lR_1) \ne lR_2$ for any $\rho \in \gal(\bar{K}/K)$. Since 
$lR_1, lR_2 \in [m/l]^{-1}P$, Galois does not act transitively on $[m/l]^{-1}P$ and so we are in case ii.
\end{proof}

Let $D_m$ be the square-free polynomial whose roots are the $x$-coordinates of the points of order $m$. Then the action of Galois on $E[m]$ is given by the Galois group of $D_m$. Note that, for $m=4$, all of the cases in Theorem \ref{nec} are necessary. For example, taking the curve ``117a4" with $P=(8,36)$ we see that iv is false because the Galois groups of $\delta_4^P$ and $D_4$ have different orders; moreover, only iii is true. For the curve ``55696ba1" and the generator Cremona gives, by checking that the curve has a trivial isogeny class, we see that only iv is true. When $m$ has two coprime proper divisors we have    

\begin{thm}
Suppose that $m>1$ is composite and $m=d_1d_2$ where $d_1, d_2$ are coprime proper divisors. If $\delta_m^P$ factorizes then either $\delta_{d_1}^P$ or $\delta_{d_2}^P$ factorizes. 
\end{thm}

\begin{proof}
There exists $x,y \in \mathbb{Z}$ such that $xd_1+yd_2=1$. Consider the homomorphism $\mathcal{G}_m \to \mathcal{G}_{d_1} \times \mathcal{G}_{d_2}$ given by $\rho \to (\rho, \rho)$. If $\rho$ is in the kernel of this map then $\rho(d_2R)=d_2R$ and $\rho(d_1R)=d_1R$ for all $R \in [m]^{-1}P$. But then $x\rho(d_1R)+y\rho(d_2R)=\rho(R)=R$ for all $R \in [m]^{-1}P$. So  $\mathcal{G}_m \cong \mathcal{G}_{d_1} \times \mathcal{G}_{d_2}$. Assume that $P \notin E[2]$ and $\delta_{d_1}^P$ is irreducible. Then for any $R \in [m]^{-1}P$ and $T \in E[d_1]$ there exists $\sigma \in \mathcal{G}_{d_1}$ with $\sigma(d_2R)=d_2R+T$. Define $(\sigma, \id)$ by  $(\sigma, \id)(R)=n_2\sigma(d_2R)+n_1(d_1R)$. Since $\mathcal{G}_m \cong \mathcal{G}_{d_1} \times \mathcal{G}_{d_2}$,  $(\sigma, \id) \in \mathcal{G}_m$.  For any $R \in [m]^{-1}P$, $(\sigma,\id)(R)=R+n_2T$. So, since $d_1$ and 
$n_2$ are coprime, $R$ and $R+T$ must correspond to roots of the same polynomial. Suppose that $\delta_m^P$ factorizes and let $R_1, R_2 \in [m]^{-1}P$ correspond to roots of two different factors. Then $\rho(R_1)-R_2 \notin E[d_1]$ or $\rho(d_1R_1) \ne \rho(d_1R_2)$  for all $\rho \in \gal(\bar{K}/K)$. Since $d_1R_1, d_1R_2 \in [d_2]^{-1}P$ it follows that $\delta_{d_2}^P$ factorizes. 
\end{proof}

Hence the case where $m$ is a composite prime power remains. Although no further results could be proven it is perhaps worth noting that, in all of Cremona's data, an example where i and ii are false in Theorem \ref{nec} could not be found when $4 < m \le 25$.

\bibliographystyle{amsplain}
\bibliography{myrefs}

\end{document}